\documentclass[12pt, reqno]{amsart}
\usepackage{amsmath, amsthm, amscd, amsfonts, amssymb, graphicx, color}
\usepackage{mathrsfs}
\usepackage[bookmarksnumbered, colorlinks, plainpages]{hyperref}
\usepackage{tikz}
\usepackage{xcolor}
\usepackage{subcaption}
\usetikzlibrary{decorations.pathreplacing,angles,quotes}

\definecolor{cof}{RGB}{219,144,71}
\definecolor{pur}{RGB}{186,146,162}
\definecolor{greeo}{RGB}{91,173,69}
\definecolor{greet}{RGB}{52,111,72}
\usepackage{xcolor}
\usepackage{enumerate}
\hypersetup{colorlinks=true,linkcolor=red, anchorcolor=green, citecolor=cyan, urlcolor=red, filecolor=magenta, pdftoolbar=true}

\newtheorem{theorem}{Theorem}[section]
\newtheorem{lemma}[theorem]{Lemma}

\newtheorem{cor}[theorem]{Corollary}
\newtheorem{definition}[theorem]{Definition}

\theoremstyle{remark}
\newtheorem{remark}[theorem]{\bf{Remark}}
\numberwithin{equation}{section}
\allowdisplaybreaks
\begin{document}

\title[$q$-Berezin sectorial operators and $q$-Berezin ranges]{$q$-Berezin sectorial operators with applications to $q$-Berezin number inequalities and  $q$-Berezin ranges}

\author[S. Mahapatra, S. Halder,  R. Birbonshi, A. Patra  and K. Paul]{Saikat Mahapatra,  Subhadip Halder,  Riddhick Birbonshi, Arnab Patra and Kallol Paul }
	
\address[Mahapatra]{Department of Mathematics, Jadavpur University, Kolkata 700032, West Bengal, India.}
\email{smpatra.lal2@gmail.com}

\address[Halder]{Department of Mathematics, Jadavpur University, Kolkata 700032, West Bengal, India.}
\email{subhadiphalderju@gmail.com }

\address[Birbonshi] {Department of Mathematics, Jadavpur University, Kolkata 700032, West Bengal, India.}
\email{riddhick.math@gmail.com}

\address[Patra] {Department of Mathematics,
Indian Institute of Technology Bhilai, 491001,
Chhattisgarh, India.}
\email{arnabp@iitbhilai.ac.in}

\address[Paul] {Vice-Chancellor, Kalyani University, West Bengal 741235 and 
		Professor (on lien), Department of Mathematics, Jadavpur University, Kolkata 700032, West Bengal, India.}
	\email{kalloldada@gmail.com}



\maketitle	

\begin{abstract}  
In this paper, we introduce a new class of operators, called $q$-Berezin sectorial operators, as an extension of the class of $q$-sectorial operators. By presenting examples on the Hardy–Hilbert space, we show that there exist operators which are $q$-Berezin sectorial but not $q$-sectorial. Several new inequalities for the $q$-Berezin number associated with this class of operators are also derived. In addition, we investigate the geometric structure of the $q$-Berezin range for various classes of operators on the Bergman space, including weighted shift and certain composition operators.
 
\end{abstract}

\tableofcontents

\section{Introduction}

 A reproducing kernel Hilbert space $ \mathscr{H}$ over a nonempty set $\Omega$ is a Hilbert space of complex-valued functions on $\Omega$ such that, for every $\lambda \in \Omega$,  the linear evaluation functional on $\mathscr{H}$ given by $\phi \to \phi(\lambda),$ is continuous (see \cite{PR_Book_2016}). 
By the Riesz representation theorem, for each $\lambda \in \Omega$, there exists a unique $ k_\lambda \in \mathscr{H}$ satisfying $\phi(\lambda)=\langle \phi, k_\lambda \rangle,$ for all $\phi \in  \mathscr{H}$. The family $\{k_\lambda :  \lambda \in \Omega \}$ forms the reproducing kernels of $\mathscr{H}$, and the normalized kernels are $\{\hat{k}_{\lambda}=k_\lambda/\|k_\lambda\| :  \lambda\in \Omega\}.$ 
 For  $T \in  \mathscr{B}( \mathscr{H}),$ the space of all bounded linear operators on $\mathscr{H},$ the Berezin transform of $T$ is defined as
$$\widetilde{T}(\lambda)=\langle T\hat{k}_{\lambda},\hat{k}_{\lambda} \rangle~~\text{for all $\lambda \in\Omega$},$$ (see \cite{covariant,BER1}).  The Berezin range of $T$ is defined by 
$\textbf{Ber}(T)= \{\widetilde{T}(\lambda) : \lambda \in \Omega\},$
(see \cite{Reproducing}). For $T \in  \mathscr{B}( \mathscr{H}),$ the Berezin number and Berezin norm are defined by  
$$\textbf{ber}(T)=\sup \left\{|\widetilde{T}(\lambda)| : \lambda \in\Omega \right\}\,\mbox{and}\, \|T\|_{ber}=\sup\left\{|\langle T\hat{k}_{\lambda},\hat{k}_{\mu}\rangle | : \lambda,\mu \in \Omega\right\},$$
(see \cite{BY_JIA_2020,Berezin symbol}). Moreover, $\textbf{Ber}(T)\subseteq W(T)$ and hence $\textbf{ber}(T)\le w(T)$, where
$W(T)$ and $w(T)$ denote the numerical range and numerical radius of $T$ , respectively. It was further shown in \cite{pintuani} that, for positive operators, the Berezin number and the Berezin norm are identical. We also refer the reader to \cite{pintuRacsam,Barik,mahapatra,Majee} for recent developments concerning inequalities involving the Berezin number.

      The concepts of the $q$-Berezin range and $q$-Berezin number for $q\in(0,1]$ are introduced by Stojiljković et al. in \cite{filomatq}. For an operator $T\in\mathscr{B}(\mathscr{H})$, these are defined by \begin{align*}
    \textbf{Ber}_q(T)=&\left\{\langle T \hat{k}_{\lambda},\hat{k}_{\mu}\rangle: \lambda,\mu\in \Omega,\, \langle \hat{k}_{\lambda},\hat{k}_{\mu}\rangle=q\right\}\\
    &\,\,\,\,\,\,\,\,\,\,\,\,\,\,\,\,\,\,\,\,\,\,\,\,\,\,\mbox{and}\\\textbf{ber}_q(T)=&\sup_{ \lambda,\mu\in \Omega}\left\{|\langle T \hat{k}_{\lambda},\hat{k}_{\mu}\rangle|: \langle \hat{k}_{\lambda},\hat{k}_{\mu}\rangle=q\right\}.
\end{align*}
Clearly, $\mathbf{Ber}_q(T)$ is a bounded and nonempty subset of $\mathbb{C}$, and it satisfies the following properties:
\begin{enumerate}[\upshape (i)]
    \item $\textbf{Ber}_q(aT+bI)=a  \textbf{Ber}_q(T)+qb$ for all $a,b \in \mathbb{C}$
    \item $\textbf{Ber}_q(T)\subseteq W_q(T)$
    \item $\textbf{Ber}_q(T^*)=\overline{  \textbf{Ber}_q(T)}$,
\end{enumerate}
 where $W_q(T)$ denotes the $q$-numerical range of $T$. Moreover, when $q=1$, then  $q$-Berezin range and $q$-Berezin number of a bounded linear operator $T$ reduce to the classical Berezin range and classical Berezin number of $T$, respectively, i.e.,  $\textbf{Ber}_q(T)=\textbf{Ber}(T)$ and $\textbf{ber}_q(T)=\textbf{ber}(T)$.   Furthermore, for every $T\in \mathscr{B(H)}$, the inequality $\textbf{ber}_q(T)\le w_q(T)$, where $w_q(T)$ denotes the $q$-numerical radius of $T$, is satisfied. For further details on the $q$-Berezin range, we refer the reader to \cite{Bhattacharya,neq,filomatq}.

An operator is called sectorial whenever its numerical range is contained in a sector of the right half of the complex plane with the origin as its vertex (see \cite{kato}). In this direction, Mahapatra et al. \cite{saikat} introduced the notion of Berezin sectorial operators, where the Berezin range is assumed to lie within such a sector. This framework has proven useful in obtaining sharper estimates for the Berezin number. 
Motivated by these developments, particularly the concepts of Berezin sectorial operators and $q$-Berezin structures, we introduce a new class of operators, namely $q$-Berezin sectorial operators. This class generalizes the family of  $q$-sectorial operators, where a bounded linear operator is said to be $q$-sectorial if its $q$-numerical range is contained in a sector (see \cite{Kittaneh}).

The Berezin range, unlike the numerical range, is not necessarily convex. The study of the convexity of the Berezin range for composition operators was initiated in \cite{convexity} and subsequently extended to several other classes of operators; see \cite{Composition,Bulletin des,Reproducing,berezin toeplitz}. This naturally raises the question of the convexity of the $q$-numerical range and the $q$-Berezin range for a bounded linear operator. It is known that the $ q$-numerical range of a bounded linear operator is convex; see \cite{gau}.   More recently,   Bhattacharya and Patra \cite{Bhattacharya} have investigated the geometric structure of the $q$-Berezin range for various classes of operators on the Hardy-Hilbert space . Motivated by these developments, we examine the geometric structure of the $q$-Berezin range for several classes of operators on the Bergman Space.

This article is organized as follows. Including the introduction, it consists of two sections. Section \ref{s2} introduces a new class of operators, termed $q$-Berezin sectorial operators, which extends the set of all  $q$-sectorial operators. We present several examples on Hardy–Hilbert space to illustrate the motivation and scope of this class.  Through examples on the Hardy–Hilbert space, we show the existence of operators that are $q$-Berezin sectorial but not $q$-sectorial. We also show that the $q$-Berezin sectorial index may be strictly smaller than the classical $q$-sectorial index. These observations provide a  motivation for the study of $q$-Berezin number inequalities for $q$-Berezin sectorial operators from a geometric perspective and allows
to develop $q$-Berezin number inequalities to operators that are not classically $q$-sectorial. We derive several bounds for the $q$-Berezin number associated with these operators, improving and extending several known results.
In Section \ref{s3}, we investigate the convexity of the q-Berezin range and its symmetry with respect to the real and imaginary axes for several classes of operators, including weighted shift operators and composition operators on the Bergman space.

\section{$q$-Berezin sectorial operators }\label{s2}
In this section, we first introduce  $q$-Berezin  sectorial operators, and to do so, we first  recall the definition of a $q$-sectorial matrix from \cite{Kittaneh}.

  \begin{definition} Let  $\theta\in[0,\frac{\pi}{2})$, and consider the sector $S_\theta=\{z\in\mathbb{C}: \,\, \Re(z)>0 \, \mbox{and}\, |\Im(z)|\le \tan \theta\,\Re(z) \},$
 with vertex at the origin and semi-angle $\theta$. An $n\times$n  matrix $M$  is said to be $q$-sectorial, where $q\in (0,1]$, if its $q$-numerical range is contained entirely in the sector $S_{\theta}$, that is, $W_q(M)\subseteq S_\theta$.   The smallest value of $\theta$ for which $W_q(M)\subseteq S_\theta$ is called the $q$-sectorial index of $M$.
\end{definition}

 For any $z\in \mathbb{C}$ and  $\theta\in [0,\frac{\pi}{2})$, we have  $|\arg(z)|\le \theta $ if and only if $\Re(z)>0$ and $|\Im(z)|\le \tan \theta\,\Re(z)$. Now, for $q\in(0,1]$, an operator  $T\in \mathscr{B(H)}$ is called $q$-sectorial operator  if  its $q$-numerical range satisfies $W_q(T)\subseteq S_\theta$, where $\theta\in[0,\frac{\pi}{2})$.

\begin{definition}
 Let $q\in(0,1]$ and  $\theta\in[0,\frac{\pi}{2})$. Consider the sector
  $S_\theta=\{z\in\mathbb{C}: \,\, |\arg z|\le \theta\}$.  A linear operator $T$ acting on a
 reproducing kernel Hilbert space $\mathscr{H}$ is said to be  $q$-Berezin sectorial with vertex at origin and the semi-angle $\theta\in[0,\frac{\pi}{2})$, if its $q$-Berezin range satisfies $\textbf{Ber}_q(T)\subseteq S_\theta$, that is, the $q$-Berezin range of $T$ lies entirely within the sector $S_{\theta}$
	of the right half-plane having vertex at the origin and semi-angle $\theta$. 
 The collection of all such operators will be denoted by $\Pi^{\textbf{Ber}_q}_{\theta}.$ The smallest value of $\theta$ for which $\textbf{Ber}_q(T)\subseteq S_\theta$ is called the $q$-Berezin sectorial index of $T$.
\end{definition}

 In particular, when $q=1$, the notion of a $q$-Berezin sectorial operator coincides exactly with the class of Berezin sectorial operators introduced in \cite{saikat}. Hence, the class of $q$-Berezin sectorial operators provides a natural extension of Berezin sectorial operators.
Moreover, it follows clearly from the definitions that for any $q\in(0,1]$, every $q$-sectorial operator is a $q$-Berezin sectorial operator. However, the converse is not true; that is, not every $q$-Berezin sectorial operator is $q$-sectorial.

To illustrate the notion of $q$-Berezin sectorial operators, we begin with the multiplication operator on Hardy–Hilbert space. 
The Hardy–Hilbert space over the open unit disk $\mathbb{D}$, denoted by $H^2(\mathbb{D})$, consists of all analytic functions whose Taylor coefficients are square-summable, that is,
$$H^2(\mathbb D)=\left\{f : f(z)=\sum_{n=0}^{\infty}a_nz^n~~\text {with}~~ \sum_{n=0}^{\infty}|a_n|^2 < \infty\right\}.$$ 
The space $H^2(\mathbb D)$ is a reproducing kernel Hilbert space. Its reproducing kernel, known as the Szeg\H{o} kernel, is defined for each  $w \in \mathbb D$  by
$$k_w(z)=\frac{1}{1-\bar wz}~~\text{for all $z \in \mathbb D$}.$$
 The set of multipliers on $H^2(\mathbb D)$ is defined as 
$$ \mathcal{M}(H^2(\mathbb{D}))=\big\{g\in H^2(\mathbb{D}): gf \in H^2(\mathbb{D}), \,\,\mbox{for all} \,f\in H^2(\mathbb{D})\big\}. $$ It is well-known that $\mathcal{M}(H^2(\mathbb{D}))=H^{\infty}(\mathbb{D})$. For $g\in H^{\infty}(\mathbb{D})$, the multiplication operator $M_g$ is defined by \[M_g(f)(z)=g(z)f(z).\] 
Let $g(z)=az+b$, where $a,b\in \mathbb{C}$, and $q\in(0,1]$. It was shown in \cite{Bhattacharya} that the $q$-Berezin range of the multiplication operator $M_{az+b}$ is the open disk centered at $bq$ and radius $|a|q$.\\
On the other hand,  since $M_z$ acts as the unilateral shift on $H^2(\mathbb D)$, it follows from \cite{gau} that the   $q$-numerical range  $W_q(M_{az+b})$  is the closed circular disc centered at $bq$ with radius $|a|$ for $q\in(0,1)$. In the case $q=1$, $W_q(M_{az+b})$  is the open disk centered at $b$ with radius $|a|$.

For the multiplication operator $M_{\phi(z)}$ with symbol $\phi(z)=(1+i)(1+z)$ and  $q=\frac{1}{2}$, Figure 1 provides the $q$-Berezin range as an open disc contained within the $q$-numerical range, which is a closed disc.

\newpage
 
\begin{figure}[h]
\centering
\includegraphics[width=8cm]{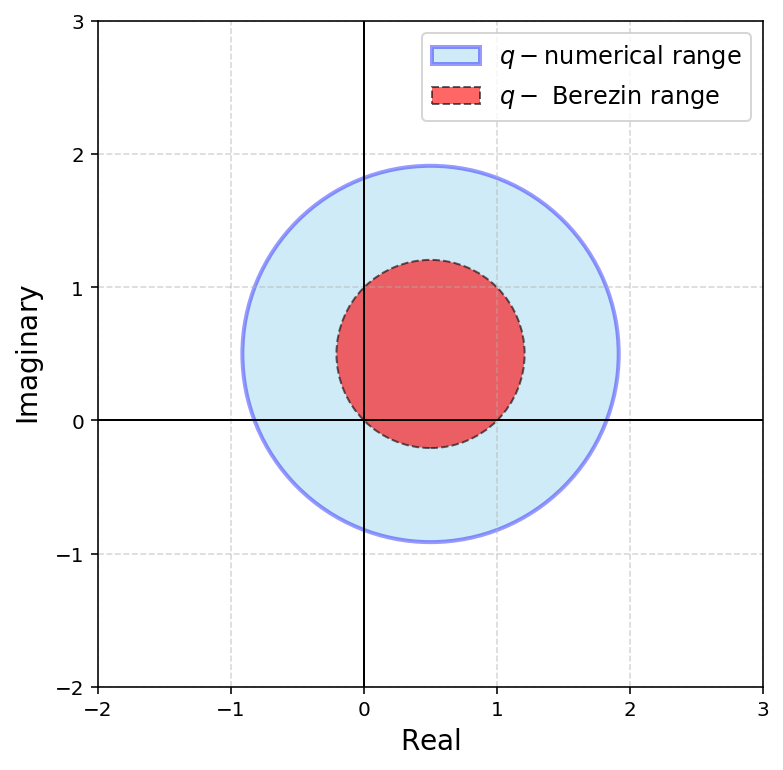}
\caption{$\textbf{Ber}_q(M_{\phi(z)})$ and $W_q(M_{\phi(z)})$ for $q=\frac{1}{2}$ and $\phi(z)=(1+i)(1+z)$. }
\label{fig:sample_plot}
\end{figure}

Figure 2 provides the same for the operator $M_{\phi(z)}+I$, where $\phi(z)=(1+i)(1+z)$ and  $q=\frac{1}{2}$. From the figure, it is clear that this operator is not $q$-sectorial. However $\textbf{Ber}_q(M_{\phi(z)}+I)\subseteq S_{\theta}$, where $\theta\approx\frac{\pi}{2.735}$. Therefore it is $q$-Berezin sectorial with index approximately $\frac{\pi}{2.735}.$

\begin{figure}[h]
\centering
\includegraphics[width=8cm]{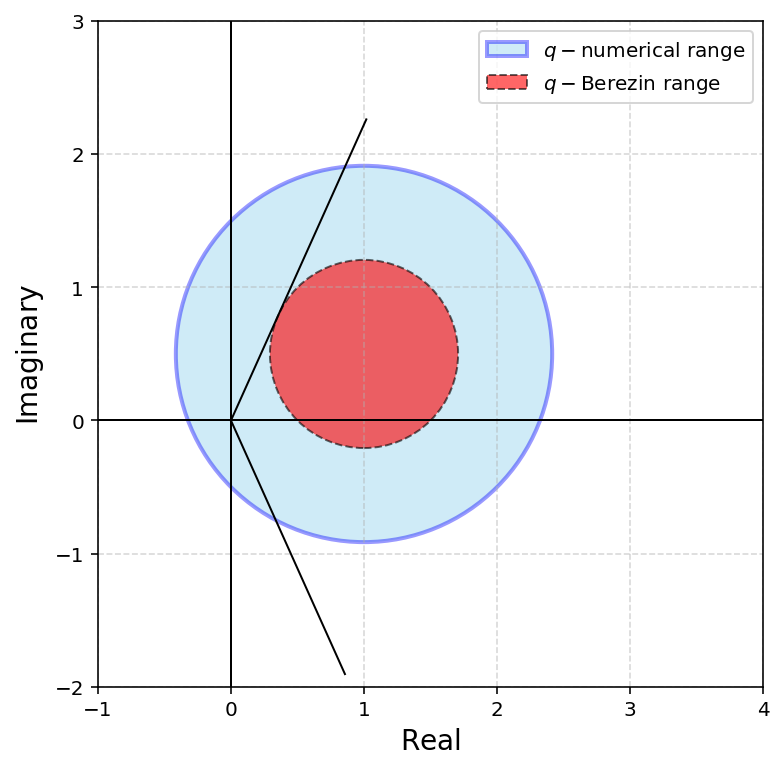}
\caption{$\textbf{Ber}_q(M_{\phi(z)}+I)$ and $W_q(M_{\phi(z)}+I)$ for $q=\frac{1}{2}$ and $\phi(z)=(1+i)(1+z)$. }
\label{fig:sample_plot}
\end{figure}

Figure 3 provides the same for the operator $M_{\phi(z)}+\frac{5}{2}I$, where $\phi(z)=(1+i)(1+z)$ and  $q=\frac{1}{2}$. From the figure, it is clear that this operator is $q$-sectorial with $q$-sectorial index approximately $\frac{\pi}{2.688}$ and is also $q$-Berezin sectorial with index approximately $\frac{\pi}{4.637}.$

\begin{figure}[h]
\centering
\includegraphics[width=8cm]{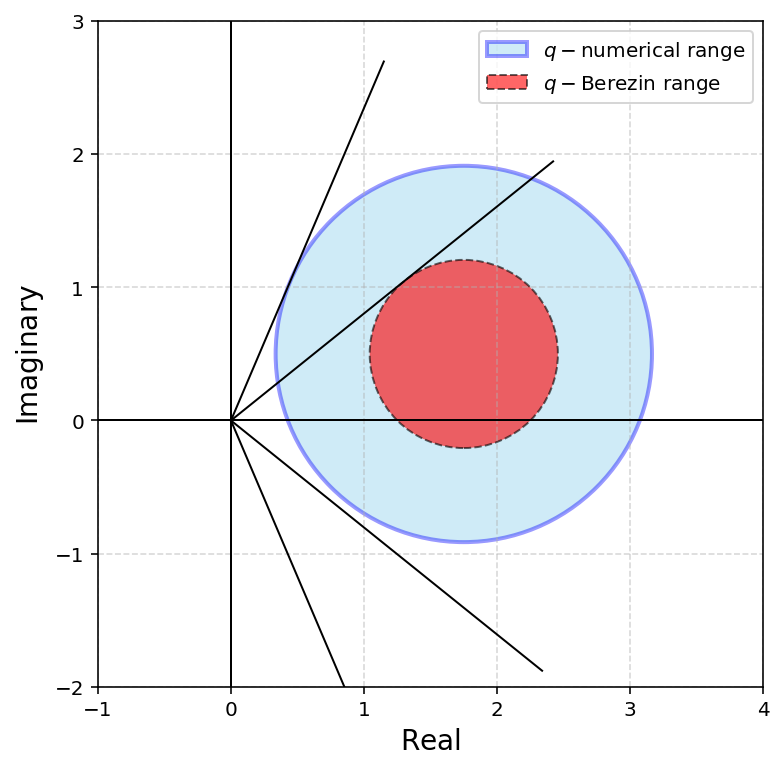}
\caption{$\textbf{Ber}_q(M_{\phi(z)}+\frac{5}{2}I)$ and $W_q(M_{\phi(z)}+\frac{5}{2}I)$ for $q=\frac{1}{2}$ and $\phi(z)=(1+i)(1+z)$. }
\label{fig:sample_plot}
\end{figure}

These examples show that operators which are not $q$-sectorial can still be represented as $q$-Berezin sectorial operators. Moreover, the $q$-Berezin sectorial index is often strictly smaller than the $q$-sectorial index, as seen in the case of the operator $M_{\phi(z)}+\frac{5}{2}I$, where $\phi(z)=(1+i)(1+z)$ (see Figure 3). As noted earlier, in this case the  $q$-sectorial index is approximately $\frac{\pi}{2.688}$, whereas the $q$-Berezin sectorial index is reduced to approximately $\frac{\pi}{4.637}$. The improved bounds of the $q$-numerical radius for $q$-sectorial matrices obtained in \cite{Kittaneh}, together with the geometric viewpoint of $q$-Berezin sectorial operators introduced above, naturally motivate the development of sharper bounds for the $q$-Berezin number of $q$-Berezin sectorial operators.

Any operator $T\in \mathscr{B(H)}$ possesses a unique representation $T=\Re(T)+i\Im(T)$, where $\Re(T)=\frac{1}{2}(T+T^*)$ and $\Im(T)=\frac{1}{2i}(T-T^*)$. This representation is known as the Cartesian decomposition of $T$.
Next, we establish several inequalities involving the $q$-Berezin number of  $q$-Berezin sectorial operators. We first present the following lemma, which is used to obtain the subsequent results.
\begin{lemma}\label{ber(im)}
    Let $q\in(0,1)$ and $T\in \Pi_{\theta}^{\textbf{Ber}_q}$. Then\label{main}
    \[\left( \frac{\sin\theta}{q}+\frac{4\sqrt{1-q^2}}{q^2}\right)\,\textbf{ber}_q(T)\ge\textbf{ber}(\Im(T)). \]
\end{lemma}
\begin{proof}
Let  $\theta\neq 0$ and let $\hat{k}_{\lambda}$ and  $\hat{k}_{\mu}$ be the normalized reproducing kernels of $\mathscr{H}$ such that  $\langle \hat{k}_{\lambda},\hat{k}_{\mu}\rangle=q$. Then
     \begin{eqnarray}
        \textbf{ber}_q(T)&\ge& \big|\langle T\hat{k}_{\lambda},\hat{k}_{\mu} \rangle \big|\nonumber\\& = & \sqrt{ \big(\Re \langle T\hat{k}_{\lambda},\hat{k}_{\mu}\rangle\big)^2+ \big(\Im\langle T\hat{k}_{\lambda},\hat{k}_{\mu}\rangle\big)^2}\nonumber\\
         & \ge &\sqrt{ \cot^2\theta\,\big(\Im\langle T\hat{k}_{\lambda},\hat{k}_{\mu}\rangle\big)^2+\big(\Im\langle T\hat{k}_{\lambda},\hat{k}_{\mu}\rangle\big)^2}\nonumber\\
          & = & \csc\theta\,\big|\Im\langle T\hat{k}_{\lambda},\hat{k}_{\mu}\rangle\big|\label{yytt}.
     \end{eqnarray} Now, by taking $\hat{k}_{\mu}=q\hat{k}_{\lambda}+\sqrt{1-q^2}v$, for some unit vector $v$   orthogonal to  $\hat{k}_{\lambda}$, we obtain
      \begin{eqnarray}
        \sin\theta\, \textbf{ber}_q(T)&\ge& \big|\Im\langle T\hat{k}_{\lambda},\hat{k}_{\mu}\rangle\big|\nonumber\\
        &=& \big|q\Im\langle T\hat{k}_{\lambda},\hat{k}_{\lambda}\rangle+\sqrt{1-q^2}\Im\langle T\hat{k}_{\lambda},v\rangle\big|\nonumber\\
        &\ge& q\big|\Im\langle T\hat{k}_{\lambda},\hat{k}_{\lambda}\rangle\big|-\sqrt{1-q^2}\big|\Im\langle T\hat{k}_{\lambda},v\rangle\big|\nonumber\\
             &\ge& q\big|\langle \Im(T)\hat{k}_{\lambda},\hat{k}_{\lambda}\rangle\big|-\sqrt{1-q^2}\|T\|\nonumber\\
                  &\ge&  q\big|\langle \Im(T)\hat{k}_{\lambda},\hat{k}_{\lambda}\rangle\big|-\frac{4\sqrt{1-q^2}}{q}\textbf{ber}_q(T)\label{ppaa},
     \end{eqnarray} where the last inequality comes from the fact that $\frac{q}{4}\|T\|\le\textbf{ber}_q(T)$. Finally, taking the supremum over all $\lambda\in\Omega$, we arrive at the desired inequality.\\
If $\theta=0$, then $|\Im\langle T\hat{k}_{\lambda},\hat{k}_{\mu}\rangle|=0$. On the other hand, we have $$ \big|\Im\langle T\hat{k}_{\lambda},\hat{k}_{\mu}\rangle\big|\ge  q\big|\langle \Im(T)\hat{k}_{\lambda},\hat{k}_{\lambda}\rangle\big|-\frac{4\sqrt{1-q^2}}{q}\textbf{ber}_q(T).$$ Combining these two relations yields the desired result.
\end{proof} 

From the definition of the $q$-Berezin symbol and the Cauchy-Schwarz inequality, it follows that $$|\langle T\hat{k}_{\lambda},\hat{k}_{\mu}\rangle|^2\le|\langle T^*T\hat{k}_{\lambda}, \hat{k}_{\lambda}\rangle| \le \||T|^2\|_{ber},$$ where $|T|^2=T^*T$. Consequently,
 $\textbf{ber}_q(T)\le \||T|^2\|_{ber}$ for any operator $T\in \mathscr{B(H)}$. Motivated by this inequality, our objective in the following result is to establish a new reverse-type inequality between the $q$-Berezin number $\textbf{ber}_q(T)$   and the Berezin norm $\||T|^2\|_{ber}$.
\begin{theorem}\label{last31} Let  $q\in (0,1)$,   $T\in\Pi_{\theta}^{\textbf{Ber}_q}$ and $\alpha\in \mathbb{R}$. Then
 \begin{align*}
  \||T|^2\|_{ber}\le &\frac{1}{4}(1+\sin\theta)^2\textbf{ber}_q^2(T) +\frac{1}{2}\inf_{\alpha\in \mathbb{R}}\Big\{\Big(\sqrt{\||T-\alpha qI|^2\|_{ber}}+|\alpha|\sqrt{1-q^2}\Big)^2\\
&+\Big(\sqrt{\||T-i\alpha qI|^2\|_{ber}}+|\alpha|\sqrt{1-q^2}\Big)^2\Big\}.\end{align*}    
\end{theorem}
\begin{proof} To establish the result, we first recall the following inequality from \cite{potpourri}:
   \begin{eqnarray}
      0\le \|u\|^2\|v\|^2\le\frac{1}{4}\Big(\Im\langle u,v \rangle+\Re\langle u,v \rangle\Big)^2+  \frac{\|v\|^2}{2}\Big(\|u-\alpha v\|^2+\|u-i\alpha v\|^2\Big),\label{last30}
    \end{eqnarray} for all $u,v\in \mathscr{H}$ and  $\alpha\in \mathbb{R}$. 

     Let $\hat{k}_{\lambda}$ and  $\hat{k}_{\mu}$ be the normalized reproducing kernels of $\mathscr{H}$ such that  $\langle \hat{k}_{\lambda},\hat{k}_{\mu}\rangle=q$ and $\alpha\in \mathbb{R}$. Putting $u=T\hat{k}_{\lambda}$ and $v=\hat{k}_{\mu}$ in \eqref{last30}, we obtain
      \begin{align}
\|T\hat{k}_{\lambda}\|^2&\le \frac{1}{4}\Big(\Im\langle T\hat{k}_{\lambda},\hat{k}_{\mu} \rangle+\Re\langle T\hat{k}_{\lambda},\hat{k}_{\mu} \rangle\Big)^2+\frac{1}{2}\Big(\|T\hat{k}_{\lambda}-\alpha\hat{k}_{\mu}\|^2+\|T\hat{k}_{\lambda}-i\alpha\hat{k}_{\mu}\|^2\Big)\nonumber\\
&\le \frac{1}{4}\Big(|\Im\langle T\hat{k}_{\lambda},\hat{k}_{\mu} \rangle|+|\Re\langle T\hat{k}_{\lambda},\hat{k}_{\mu} \rangle|\Big)^2+\frac{1}{2}\Big(\|T\hat{k}_{\lambda}-\alpha\hat{k}_{\mu}\|^2+\|T\hat{k}_{\lambda}-i\alpha\hat{k}_{\mu}\|^2\Big)\nonumber\\
&\le \frac{1}{4}(1+\sin\theta)^2\textbf{ber}_q^2(T) +\frac{1}{2}\Big(\|T\hat{k}_{\lambda}-\alpha\hat{k}_{\mu}\|^2+\|T\hat{k}_{\lambda}-i\alpha\hat{k}_{\mu}\|^2\Big)\,\,(\mbox{by \eqref{yytt}}).\label{kkii}
    \end{align}
 Now, by taking $\hat{k}_{\mu}=q\hat{k}_{\lambda}+\sqrt{1-q^2}v$, for some unit vector $v$   orthogonal to  $\hat{k}_{\lambda}$, we obtain
\begin{align*}
    \|T\hat{k}_{\lambda}-\alpha\hat{k}_{\mu}\|^2&=\Big(\|T\hat{k}_{\lambda}-\alpha q\hat{k}_{\lambda}-\alpha\sqrt{1-q^2}v\|\Big)^2\\
    &\le\Big(\|T\hat{k}_{\lambda}-\alpha q\hat{k}_{\lambda}\|+|\alpha|\sqrt{1-q^2}\|v\|\Big)^2\\
      &\le\Big(\sqrt{\||T-\alpha qI|^2\|_{ber}}+|\alpha|\sqrt{1-q^2}\Big)^2.
\end{align*}
 A similar argument gives
\begin{align*}
    \|T\hat{k}_{\lambda}-i\alpha\hat{k}_{\mu}\|^2\le\Big(\sqrt{\||T-i\alpha qI|^2\|_{ber}}+|\alpha|\sqrt{1-q^2}\Big)^2.
\end{align*}
Substituting these estimates  into the  inequality \eqref{kkii}, we obtain
 \begin{align*}
\|T\hat{k}_{\lambda}\|^2&\le \frac{1}{4}(1+\sin\theta)^2\textbf{ber}_q^2(T) +\frac{1}{2}\Big\{\Big(\sqrt{\||T-\alpha qI|^2\|_{ber}}+|\alpha|\sqrt{1-q^2}\Big)^2\\
&+\Big(\sqrt{\||T-i\alpha qI|^2\|_{ber}}+|\alpha|\sqrt{1-q^2}\Big)^2\Big\}.\end{align*}

   Taking supremum over $\lambda\in\Omega$, we get
    \begin{align*}
  \||T|^2\|_{ber}&\le \frac{1}{4}(1+\sin\theta)^2\textbf{ber}_q^2(T) +\frac{1}{2}\Big\{\Big(\sqrt{\||T-\alpha qI|^2\|_{ber}}+|\alpha|\sqrt{1-q^2}\Big)^2\\
&+\Big(\sqrt{\||T-i\alpha qI|^2\|_{ber}}+|\alpha|\sqrt{1-q^2}\Big)^2\Big\}.\end{align*}    
  
   Since this holds for every real
   $\alpha\in\mathbb{R}$, taking the infimum over $\alpha\in \mathbb{R}$ completes the proof.
\end{proof}

In the following result, we establish an upper bound for the $q$-Berezin number of a certain positive operator acting on a reproducing kernel Hilbert space $\mathscr{H}$. 
\begin{theorem}\label{ts}
    Let $q\in(0,1)$ and $S,T\in \mathscr{B(H)}$ be such that $T^*S\in \Pi_{\theta}^{\textbf{Ber}_q}$. Then 
    \begin{align*}
      & \textbf{ ber}_q(S^*S+T^*T)\\&\le\|S-iT\|^2+2\bigg(\Big( \sin\theta+\frac{4\sqrt{1-q^2}}{q}\Big)\textbf{ber}_q( T^*S)+\sqrt{1-q^2}\| \Im(T^*S)\|\bigg).
    \end{align*}
\end{theorem}
\begin{proof}    Let $\hat{k}_{\lambda}$ and  $\hat{k}_{\mu}$ be the normalized reproducing kernels of $\mathscr{H}$ such that  $\langle \hat{k}_{\lambda},\hat{k}_{\mu}\rangle=q$. Then we can decompose $\hat{k}_{\mu}$ as $\hat{k}_{\mu}=q\hat{k}_{\lambda}+\sqrt{1-q^2}v$, where  $v$ is a unit vector    orthogonal to  $\hat{k}_{\lambda}$. Then we have
    \begin{align}
       \big \langle (S-iT)^*(S-iT)\hat{k}_{\lambda},\hat{k}_{\mu}\big\rangle=     \big \langle (S^*S+T^*T)\hat{k}_{\lambda},\hat{k}_{\mu}\big\rangle-2\big \langle \Im(T^*S)\hat{k}_{\lambda},\hat{k}_{\mu}\big\rangle.\label{uuf}
    \end{align}
    Consequently, we obtain\begin{align*}
        &\big|\big\langle( S^*S+T^*T)\hat{k}_{\lambda},\hat{k}_{\mu}\big\rangle\big|\\&=        \big|\big\langle( S^*S+T^*T)\hat{k}_{\lambda},\hat{k}_{\mu}\big\rangle-2\big \langle \Im(T^*S)\hat{k}_{\lambda},\hat{k}_{\mu}\big\rangle+2\big \langle \Im(T^*S)\hat{k}_{\lambda},\hat{k}_{\mu}\big\rangle\big|\\
        &\le  \big|\big\langle( S^*S+T^*T)\hat{k}_{\lambda},\hat{k}_{\mu}\big\rangle-2\big \langle \Im(T^*S)\hat{k}_{\lambda},\hat{k}_{\mu}\big\rangle\big|+2\big|\big \langle \Im(T^*S)\hat{k}_{\lambda},\hat{k}_{\mu}\big\rangle\big|\\
         &\le   \big|\big \langle (S-iT)^*(S-iT)\hat{k}_{\lambda},\hat{k}_{\mu}\big\rangle\big|+2\big|\big \langle \Im(T^*S)\hat{k}_{\lambda},q\hat{k}_{\lambda}+\sqrt{1-q^2}v\big\rangle\big|\,\,\,(\mbox{by (\ref{uuf})})\\
           &\le \textbf{ber}_q\big( (S-iT)^*(S-iT)\big)+2\Big(q\big|\big \langle \Im(T^*S)\hat{k}_{\lambda},\hat{k}_{\lambda}\big\rangle\big|+\sqrt{1-q^2}\big|\big\langle \Im(T^*S)\hat{k}_{\lambda},v\big\rangle\big|\Big)\\
            &\le\|S-iT\|^2+2\Big(q\big|\Im(\langle T^*S\hat{k}_{\lambda},\hat{k}_{\lambda}\rangle)\big|+\sqrt{1-q^2}\| \Im(T^*S)\hat{k}_{\lambda}\|\Big)\\
             &\le\|S-iT\|^2+2\bigg(\Big( \sin\theta+\frac{4\sqrt{1-q^2}}{q}\Big)\textbf{ber}_q( T^*S)+\sqrt{1-q^2}\| \Im(T^*S)\|\bigg)\,\,\,(\mbox{by \eqref{ppaa}}).
    \end{align*} Therefore, by taking the supremum over all $\lambda,\mu \in \Omega$, we arrive at the required inequality.
\end{proof}

\begin{remark} In \cite[Th. 3.5]{neq}, it was shown that for $\mathcal{X},\mathcal{Y}\in \mathscr{B(H)}$ then \begin{align*}
\textbf{ber}_q(\mathcal{X}^*\mathcal{X}+\mathcal{Y}^*\mathcal{Y})\le \|\mathcal{X}-\mathcal{Y}\|^2+2\textbf{ber}_q(\mathcal{Y}^*\mathcal{X}).
\end{align*}
Substituting $\mathcal{X}=S$ and $\mathcal{Y}=iT$ into the above inequality yields
\begin{align}
       \textbf{ ber}_q(S^*S+T^*T)\le\|S-iT\|^2+2 \textbf{ber}_q( T^*S).\label{tsw}
\end{align}
    Now, consider $\frac{6}{\sqrt{37}}<q<1$ and $0<\theta< \sin^{-1}\left(1-\frac{6\sqrt{1-q^2}}{q}\right)$. Under these assumptions, we deduce that
    \begin{align*}
       & \Big( \sin\theta+\frac{4\sqrt{1-q^2}}{q}\Big) \textbf{ber}_q( T^*S)+\sqrt{1-q^2}\| \Im(T^*S)\|\\
       &\le \Big( \sin\theta+\frac{4\sqrt{1-q^2}}{q}\Big) \textbf{ber}_q( T^*S)+\frac{2\sqrt{1-q^2}}{q} \textbf{ber}_q( \Im(T^*S))\\
      & \le\left(q\sin \theta +\frac{6\sqrt{1-q^2}}{q}\right)\textbf{ber}_q( T^*S)\\
       & < \textbf{ber}_q( T^*S). 
    \end{align*} This shows that the estimate obtained in Theorem \ref{ts} yields a sharper upper bound than the bound given in \eqref{tsw}.
\end{remark}

For an operator $T\in \mathscr{B(H)}$, the power inequality for the Berezin number, \begin{align*}\textbf{ber}(T^n)\le \textbf{ber}^n(T), \,\,\,\mbox{where}\,\, n\in \mathbb{N}\end{align*} does not hold in general; see \cite[Th. 3]{coburn}. However,  it was established in \cite{garayevlama} that some classes of operators belonging to a norm-closed subalgebra also satisfy this power inequality. Moreover, \cite{garayev} established that the power inequality is valid for Toeplitz operators on the Hardy–Hilbert and Bergman spaces under suitable assumptions. More recently, in \cite{saikat}, the authors proved that if $\phi$ is harmonic on  $\mathbb{D}$ and $T_{\phi}$ is a Toeplitz operator on the Bergman space $A_\alpha^2(\mathbb{D})$, then the following power inequalities for the Berezin numbers of the real and imaginary parts of $T_{\phi}$  holds:
    \begin{eqnarray*}
        \textbf{ber}\left(\Re^n(T_{\phi})\right)\le \textbf{ber}^n\left(\Re(T_{\phi})\right)\,\,\,\,\mbox{and}\,\,  \,\,  \textbf{ber}\left(\Im^n(T_{\phi})\right)\le \textbf{ber}^n\left(\Im(T_{\phi})\right),
    \end{eqnarray*}
    for every natural number $n$. 

  Here, we establish several inequalities for the $q$-Berezin number of $q$-Berezin sectorial operators whose real and imaginary parts satisfy the Berezin number power inequality. We denote the class of all such operators by $\Pi_{\theta}^{\textbf{Ber}_q,P}$, defined as
 \begin{eqnarray*}
\Pi_{\theta}^{\textbf{Ber}_q,P}
&=&\Big\{ T\in \Pi_{\theta}^{\textbf{Ber}_q} : \textbf{ber}\left(\Re^n(T)\right)\le \textbf{ber}^n\left (\Re(T)\right),\\
&&\,\,\,\,\,\,\,\,\,\,\textbf{ber}\left(\Im^n(T)\right)\le \textbf{ber}^n \left(\Im(T)\right)\,\, \forall n\in\mathbb{N}\Big\}.     
 \end{eqnarray*}
Observe that, for a harmonic function $\phi$  on $\mathbb{D}$, there exist a constant $M >0$  such that the operator 
$T_{\phi+M}\in \Pi_{\theta}^{\textbf{Ber}_q,P}.$

\begin{theorem}
      Let $q\in(0,1)$ and $X,Y, S,T\in \mathscr{B(H)}$ with $T\in\Pi_{\theta}^{\textbf{Ber}_q,P}$. Then
          \begin{align*}
          \textbf{ber}_q(TXS\pm SYT) \le& \,2q\max\{\|XS\|,\|SY\|\}\left(\frac{4}{q^2}+\left( \frac{\sin\theta}{q}+\frac{4\sqrt{1-q^2}}{q^2}\right)^2\right)^{\frac{1}{2}}\textbf{ber}_q(T)\\&+\sqrt{1-q^2}\big\|TXS\pm SYT\big\|.
     \end{align*}
\end{theorem}

\begin{proof}  If $XS=SY=0$ then the result holds trivially. Hence, assume that $\max\{\|XS\|,\|SY\|\}\neq 0$. Let $\hat{k}_{\lambda}$ and  $\hat{k}_{\mu}$ be the normalized reproducing kernels of $\mathscr{H}$ such that  $\langle \hat{k}_{\lambda},\hat{k}_{\mu}\rangle=q$. Then we can decompose $\hat{k}_{\mu}$ as $\hat{k}_{\mu}=q\hat{k}_{\lambda}+\sqrt{1-q^2}v$, where  $v$ is a unit vector    orthogonal to  $\hat{k}_{\lambda}$.  Define  $P=\frac{XS}{\max\{\|XS\|,\|SY\|\}}$ and $Q=\frac{SY}{\max\{\|XS\|,\|SY\|\}}$. Clearly, $\|P\|\le 1$ and $\|Q\|\le 1$. Then we have
    \begin{align*}
        |\langle (TP\pm QT)\hat{k}_{\lambda},\hat{k}_{\mu}\rangle|&=|\langle (TP\pm QT)\hat{k}_{\lambda},q\hat{k}_{\lambda}+\sqrt{1-q^2}v\rangle|\\
        &\le q|\langle (TP\pm QT)\hat{k}_{\lambda},\hat{k}_{\lambda}\rangle|+\sqrt{1-q^2}|\langle (TP\pm QT)\hat{k}_{\lambda},v\rangle|\\
         &\le q\Big(|\langle P\hat{k}_{\lambda},T^*\hat{k}_{\lambda}\rangle|+|\langle T\hat{k}_{\lambda},Q^*\hat{k}_{\lambda}\rangle|\Big)+\sqrt{1-q^2}\|TP\pm QT\|\\
          &\le q\Big(\|T^*\hat{k}_{\lambda}\|+\| T\hat{k}_{\lambda}\|\Big)+\sqrt{1-q^2}\|TP\pm QT\|\\
           &\le q\sqrt{2}\Big(\|T^*\hat{k}_{\lambda}\|^2+\| T\hat{k}_{\lambda}\|^2\Big)^{\frac{1}{2}}+\sqrt{1-q^2}\|TP\pm QT\|\\
              &\le q\sqrt{2}\sqrt{\textbf{ber}(TT^*+ T^*T)}+\sqrt{1-q^2}\big\|TP\pm QT\big\|\\
              &= 2q\sqrt{\textbf{ber}(\Re^2(T)+\Im^2(T)\big)}+\sqrt{1-q^2}\big\|TP\pm QT\big\|\\
               &\le 2q\sqrt{\textbf{ber}^2(\Re(T))+\textbf{ber}^2(\Im(T))}+\sqrt{1-q^2}\big\|TP\pm QT\big\|.
    \end{align*}
    Since $\Re(T)$ is a normal operator, we have $\textbf{ber}(\Re(T))\le\|\Re(T)\|\le \frac{2}{q}\textbf{ber}_q(\Re(T))\le\frac{2}{q}\textbf{ber}_q(T)$.
   Combining this estimate with Lemma \ref{ber(im)}, we obtain
     \begin{align*}
          & |\langle (TP\pm QT)\hat{k}_{\lambda},\hat{k}_{\mu}\rangle|\\& \le 2q\left(\frac{4}{q^2}+\left( \frac{\sin\theta}{q}+\frac{4\sqrt{1-q^2}}{q^2}\right)^2\right)^{\frac{1}{2}}\textbf{ber}_q(T)+\sqrt{1-q^2}\big\|TP\pm QT\big\|.
     \end{align*}
    Taking the supremum over all $\lambda,\mu \in \Omega$, we deduce that
    \begin{align*}
          \textbf{ber}_q(TP\pm QT) \le 2q\left(\frac{4}{q^2}+\left( \frac{\sin\theta}{q}+\frac{4\sqrt{1-q^2}}{q^2}\right)^2\right)^{\frac{1}{2}}\textbf{ber}_q(T)+\sqrt{1-q^2}\big\|TP\pm QT\big\|.
     \end{align*}
  Hence, it follows that
      \begin{align*}
          \textbf{ber}_q(TXS\pm SYT) \le& \,2q\max\{\|XS\|,\|SY\|\}\left(\frac{4}{q^2}+\left( \frac{\sin\theta}{q}+\frac{4\sqrt{1-q^2}}{q^2}\right)^2\right)^{\frac{1}{2}}\textbf{ber}_q(T)\\&+\sqrt{1-q^2}\big\|TXS\pm SYT\big\|.
     \end{align*} This completes the proof. 
\end{proof} The following corollary is obtained from the theorem by choosing $X=Y=I$.
\begin{cor}\label{kkk}
       Let $q\in(0,1)$ and $S ,T\in \mathscr{B(H)}$with $T\in\Pi_{\theta}^{\textbf{Ber}_q,P}$. Then
          \begin{align*}
          \textbf{ber}_q(TS\pm ST) \le& \,2q\|S\|\left(\frac{4}{q^2}+\left( \frac{\sin\theta}{q}+\frac{4\sqrt{1-q^2}}{q^2}\right)^2\right)^{\frac{1}{2}}\textbf{ber}_q(T)\\&+\sqrt{1-q^2}\big\|TS\pm ST\big\|.
     \end{align*}
\end{cor} Using the corollary \ref{kkk}, we obtain the next result.
\begin{cor}
       Let $q\in(0,1)$ and $S,T\in\Pi_{\theta}^{\textbf{Ber}_q,P}$. Then
          \begin{align*}
          \textbf{ber}_q(TS\pm ST) \le & \min\{\beta_1,\beta_2\},
     \end{align*}
     where \begin{align*}
         \beta_1&=          \,2q\|S\|\left(\frac{4}{q^2}+\left( \frac{\sin\theta}{q}+\frac{4\sqrt{1-q^2}}{q^2}\right)^2\right)^{\frac{1}{2}}\textbf{ber}_q(T)+\sqrt{1-q^2}\big\|TS\pm ST\big\|,\\
         \beta_2&=          \,2q\|T\|\left(\frac{4}{q^2}+\left( \frac{\sin\theta}{q}+\frac{4\sqrt{1-q^2}}{q^2}\right)^2\right)^{\frac{1}{2}}\textbf{ber}_q(S)+\sqrt{1-q^2}\big\|TS\pm ST\big\|.
     \end{align*}
\end{cor}
\begin{proof}
    The proof is an immediate consequence of Corollary \ref{kkk} after interchanging $S$ and $T$.
\end{proof}

The following result is the classical Hermite-Hadamard inequality for convex functions (see \cite{hada}). Let
 $\varphi:I\to \mathbb{R}$ be a convex function, where
 $I$ is a convex subset  of $\mathbb{R}$. Then, for any  $a,b\in I$ with $a<b$, we have 
\begin{align*} 
    \varphi\left(\frac{a+b}{2}\right)\le \int_{0}^1 \varphi(ta+(1-t)b)dt\le \frac{\varphi(a)+\varphi(b)}{2}.
\end{align*}
Using this inequality, we derive the following lower bound involving the $q$-Berezin number.

\begin{theorem}\label{xxpp}
     Let $q\in(0,1)$,    $T\in \Pi^{\textbf{Ber}_q,P}_\theta$ and $\varphi$ be an increasing, continuous and convex function.  Then
     \begin{align*}
         \varphi\left(\frac{\|T^*T+TT^*\|_{{ber}}}{2}\right) &\le\int_{0}^1 \varphi\left( \frac{8t}{q^2}\textbf{ber}_q^2(T)+2(1-t)\left( \frac{\sin\theta}{q}+\frac{4\sqrt{1-q^2}}{q^2}\right)^2\textbf{ber}_q^2(T)\right)dt\\
              &\le\frac{1}{2}\left(  \varphi\left(\frac{8}{q^2}\textbf{ber}_q^2(T)\right)+ \varphi\left(2\left( \frac{\sin\theta}{q}+\frac{4\sqrt{1-q^2}}{q^2}\right)^2\textbf{ber}_q^2(T)\right)\right).
     \end{align*}
\end{theorem}
\begin{proof}  Let $T\in \Pi^{\textbf{Ber}_q,P}_\theta$.  Then we have
    \begin{align*} \frac{\|T^*T+TT^*\|_{ber}}{2}&=\textbf{ber}(\Re^2(T)+\Im^2(T))\\
    &\le \textbf{ber}^2(\Re(T))+\textbf{ber}^2(\Im(T)).\end{align*}
  Since $\varphi$ is increasing, it follows that
    \begin{align*}
        \varphi\left( \frac{\|T^*T+TT^*\|_{ber}}{2}\right)\le \varphi\left(\frac{2\textbf{ber}^2(\Re(T))+2\textbf{ber}^2(\Im(T))}{2}\right).
    \end{align*}
   Applying the Hermite-Hadamard inequality, we obtain
    \begin{align*}
            \varphi\left( \frac{\|T^*T+TT^*\|_{ber}}{2}\right)\le\int_{0}^1 \varphi\left(2t\textbf{ber}^2(\Re(T))+2(1-t)\textbf{ber}^2(\Im(T))\right)dt
    \end{align*}
     Since $\Re(T)$ is a normal operator, it follows that $$\textbf{ber}(\Re(T))\le\|\Re(T)\|\le \frac{2}{q}\textbf{ber}_q(\Re(T))\le\frac{2}{q}\textbf{ber}_q(T).$$
     Combining this estimate with Lemma \ref{ber(im)}, we get
      \begin{align*}
            &\varphi\left( \frac{\|T^*T+TT^*\|_{ber}}{2}\right)\\&\le\int_{0}^1 \varphi\left( \frac{8t}{q^2}\textbf{ber}_q^2(T)+2(1-t)\left( \frac{\sin\theta}{q}+\frac{4\sqrt{1-q^2}}{q^2}\right)^2\textbf{ber}_q^2(T)\right)dt\\
            &\le\int_{0}^1t  \varphi\left(\frac{8}{q^2}\textbf{ber}_q^2(T)\right)+(1-t) \varphi\left(2\left( \frac{\sin\theta}{q}+\frac{4\sqrt{1-q^2}}{q^2}\right)^2\textbf{ber}_q^2(T)\right)dt\\
              &\le\frac{1}{2}\left(  \varphi\left(\frac{8}{q^2}\textbf{ber}_q^2(T)\right)+ \varphi\left(2\left( \frac{\sin\theta}{q}+\frac{4\sqrt{1-q^2}}{q^2}\right)^2\textbf{ber}_q^2(T)\right)\right).
    \end{align*}Hence, the proof is complete.
\end{proof}
The following corollary follows from Theorem \ref{xxpp} by setting $\varphi(t)=t$.
\begin{cor}\label{lp12}
     Let $q\in(0,1)$ and $T\in \Pi^{\textbf{Ber}_q,P}_\theta$.  Then  
     \begin{align*}
    \frac{q^4\|T^*T+TT^*\|_{{ber}}}{\left(  8q^2+ 2\left( q\sin\theta+4\sqrt{1-q^2}\right)^2\right)} &\le\textbf{ber}_q^2(T).
     \end{align*} 
\end{cor}
\begin{remark} By \cite[Th. 3.11]{filomatq}, we have the following inequality:
\begin{align}
    \frac{q^2}{16}\|T^*T+TT^*\|\le \textbf{ber}_q(T).\label{lp1} 
\end{align}
    Consider $T=\begin{pmatrix}
        5+i & 0\\ 
        0 &  4+\frac{i}{2}
    \end{pmatrix}$ acting on $\mathscr{H}=\mathbb{C}^2$ and $q=0.95$. It is clear that $T\in \Pi^{\textbf{Ber}_q,P}_{\frac{\pi}{12}}$.  In this case, inequality \eqref{lp1} yields the lower bound $2.9333\le \textbf{ber}_q^2(T)$ whereas Corollary  \ref{lp12} gives the sharper estimate $3.6233\le  \textbf{ber}_q^2(T).$ Therefore, the lower bound obtained from Corollary \ref{lp12} is tighter than the one provided by \eqref{lp1}.
\end{remark}

\section{$q$-Berezin range of operators on Bergman space}
\label{s3}

Let $\mathbb{D}=\{z\in \mathbb{C}:|z|<1\}$ denote the open unit disk, and let $\mathbb{T}_r=\{z\in \mathbb{C}:|z|=r\}$ represent the circle of radius $r$ centered at the origin. Throughout,  $dA$ denotes the area measure on $\mathbb{D}$, normalized to make the area of $\mathbb{D}$ equal to $1$.
 The Bergman space, denoted by $A^2(\mathbb{D})$, is defined by
\begin{align*}
    A^2(\mathbb{D})=\left\{f\in \mbox{Hol}(\mathbb{D)}:\int_{\mathbb{D}}|f(z)|^2dA<\infty\right\},
\end{align*} where $\mbox{Hol}(\mathbb{D)}$  denotes the collection of all holomorphic functions on $\mathbb{D}$. It is well known that $   A^2(\mathbb{D})$ forms a reproducing kernel Hilbert space. The reproducing kernel associated with $A^2(\mathbb{D})$ is given by
\begin{align*}
    k_{\lambda}(z)=\frac{1}{(1-\overline{\lambda}z)^{2}}\,\,\, \lambda,z \in \mathbb{D}
\end{align*} and  the corresponding normalized reproducing kernel is \begin{align*}
    \hat{k}_{\lambda}(z)=\frac{(1-|\lambda|^2)}{(1-\overline{\lambda}z)^{2}}\,\,\, \lambda,z \in \mathbb{D}.
\end{align*}  For more information on Bergman space, we refer the reader to \cite{zhu}.

Next, for a given $\lambda\in \mathbb{D}$, we  classify all values of $\mu$ satisfying the condition $\langle \hat{k}_{\lambda},\hat{k}_{\mu}\rangle=q$, which is fundamental to the definition of the $q$-Berezin range. Based on this classification, we further establish several structural properties of the $q$-Berezin range.

  For any $\lambda,\mu\in \mathbb{D}$,  we have\begin{align*}\langle\hat{k}_{\lambda},\hat{k}_{\mu}\rangle
  =\frac{{k}_{\lambda}(\mu)}{\|k_{\lambda}\|\|k_{\mu}\|}=\frac{(1-|\lambda|^2)(1-|\mu|^2)}{(1-\overline{\lambda}\mu)^2}.\end{align*}
Since $\lambda,\mu\in \mathbb{D}$, it follows that $\langle\hat{k}_{\lambda},\hat{k}_{\mu}\rangle\neq 0$.
Let $q\in(0,1]$  be fixed and  $\lambda\in \mathbb{D}$,  define the set
\begin{align*}
    S_{\lambda}=\left\{\mu\in \mathbb{D}: \langle \hat{k}_{\lambda},\hat{k}_{\mu}\rangle=q\right\}.
\end{align*}
We now present the following result concerning this set.
\begin{theorem}\label{kktt}
     Let $\lambda\in \mathbb{D}$ and $q\in (0,1]$. Then $S_0=\mathbb{T}_{\sqrt{1-q}}$ and  for $\lambda\neq 0$, $S_{\lambda}=\{\gamma^{+}_{\lambda}\lambda,\gamma^{-}_{\lambda}\lambda\}\subseteq \mathbb{D}$,  where $$\gamma_{\lambda}^{\pm}=\frac{q|\lambda|\pm(1-|\lambda|^2)\sqrt{1-q}}{|\lambda|\Big(1-(1-q)|\lambda|^2\Big)}.$$
\end{theorem}

\begin{proof}
    Let $\lambda=a_1+ib_1$, $\mu=a_2+ib_2\in \mathbb{D}$. Now, \begin{align}
       & \langle \hat{k}_{\lambda},\hat{k}_{\mu}\rangle=q\nonumber\\
       &\Rightarrow \frac{(1-|\lambda|^2)(1-|\mu|^2)}{(1-\overline{\lambda}\mu)^2}=q\nonumber\\
       &\Rightarrow \frac{(1-|\lambda|^2)(1-|\mu|^2)}{q}=\Big(1-\big(a_1a_2+b_1b_2+i(a_1b_2-b_1a_2)\big)\Big)^2.\label{zzii}
    \end{align} Comparing the real and imaginary parts gives $(1-a)b=0$, where $a=a_1a_2+b_1b_2$ and $b=a_1b_2-b_1a_2$. Since $a=\Re(\overline{\lambda}\mu)\le|\overline{\lambda}\mu|\le |\lambda||\mu|<1$, it follows that  $b=0$. Hence $\Im(\overline{\lambda}\mu)=0$.

    If $\lambda=0$, then $|\mu|=\sqrt{1-q}$. Therefore, $S_0= \mathbb{T}_{\sqrt{1-q}}.$

    Now, consider $\lambda\neq 0$. Since $b=0$, it implies that $\mu=\gamma\lambda$ for some $\gamma\in \mathbb{R}.$ Substituting $\mu=\gamma\lambda$ into \eqref{zzii}, we obtain
    \begin{align}
       & \frac{(1-|\lambda|^2)(1-\gamma^2|\lambda|^2)}{(1-\gamma|\lambda|^2)^2}=q\nonumber\\
      & \Rightarrow \gamma^2|\lambda|^2\Big((1-q)|\lambda|^2-1\Big)+2q\gamma|\lambda|^2+1-|\lambda|^2-q=0.\label{q1}
    \end{align} Denote the two roots of equation \eqref{q1} by $\gamma_{\lambda}^{\pm}$, i.e., $$\gamma_{\lambda}^{\pm}=\frac{q|\lambda|\pm(1-|\lambda|^2)\sqrt{1-q}}{|\lambda|\Big(1-(1-q)|\lambda|^2\Big)}.$$ Thus $\mu = \gamma^{\pm}_{\lambda}\lambda$ and consequently,  $S_{\lambda}=\{\gamma^{+}_{\lambda}\lambda,\gamma^{-}_{\lambda}\lambda\}\subseteq \mathbb{D}$.
\end{proof}
From the the condition $\langle\hat{k}_{\lambda},\hat{k}_{\mu}\rangle=q$, we obtain the following relation between $\lambda $ and $\mu$:
\begin{align}
    (1-|\lambda|^2)(1-|\mu|^2)=q(1-\overline{\lambda}\mu)^2.\label{q}
\end{align}   By Theorem \ref{kktt}, it follows that $\mu=\gamma^{\pm}_{\lambda}\lambda$ where $\gamma_{\lambda}^{\pm}=\frac{q|\lambda|\pm(1-|\lambda|^2)\sqrt{1-q}}{|\lambda|\Big(1-(1-q)|\lambda|^2\Big)}$ for $0<|\lambda|<1$. Consequently, we have $\overline{\lambda}\mu=\gamma^{\pm}_{\lambda}|\lambda|^2$. Also $\lambda=0$ implies $\overline{\lambda}\mu=0$.
The next lemma is useful in the remaining part of the paper.
\begin{lemma}\label{oopp}
    Let $q\in (0,1]$. Then
    \[\Big\{\gamma_{\lambda}^{+}|\lambda|^2:0<|\lambda|<1\Big\}\cup\Big\{\gamma_{\lambda}^{-}|\lambda|^2:0<|\lambda|<1\Big\}=\begin{cases}(0,1) & \text{if }q=1\\\left[\frac{\sqrt{q}-1}{\sqrt{q}+1},1\right)& \text{if }q\in(0,1)
     \end{cases}.\] 
\end{lemma}
\begin{proof}
     If $q=1$, then clearly $\gamma_{\lambda}^{\pm}=1$. Hence, $\big\{|\lambda|^2:0<|\lambda|<1\big\}=(0,1)$.
     
    Now suppose $q\in (0,1)$ and $0<|\lambda|<1$. It follows that $\gamma^{-}_{\lambda}|\lambda|^2\le\gamma^{+}_{\lambda}|\lambda|^2$. Let $a=\sqrt{1-q}$ and $|\lambda|=t$. Define the function $\varphi:(0,1)\to \mathbb{R}$ by 
     \begin{align*}
         \varphi(t)=\gamma^{-}_{\lambda}t^2=\frac{(1-a^2)t^2-at(1-t^2)}{1-a^2t^2}.
     \end{align*}
    Differentiating, we obtain \begin{align*}
         \varphi^{\prime}(t)=\frac{-a+2(1-a^2)t+(3a-a^3)t^2-a^3t^4}{\big(1-a^2t^2\big)^2}.
     \end{align*}
     Denote the numerator by $$f(t)=-a+2(1-a^2)t+(3a-a^3)t^2-a^3t^4.$$
Observe that
 $$\lim_{t\to 0^+}f(t)=-a<0,$$ so $f(t)<0$ near $0$.  Consequently, $\varphi^{\prime}(t)<0$ near $0$.\\
      On the other hand, $$\lim_{t\to 1^-}f(t)=2(1+a)(1-a^2)>0,$$ which implies that $f(t)>0$ near $1$. Hence, $\varphi^{\prime}(t)>0$ near $1$.

     Next, $f^{\prime}(t)=2\big(1-a^2+(3a-a^3)t-2a^3t^3\big)$. Since, $f^{\prime}(t)>0$ on $(0,1)$, the function $f$ is strictly increasing on $(0,1)$. Because $f(t)<0$ near $0$ and  $f(t)>0$ near $1$, it follows that $f$ has exactly one root $t_0\in (0,1)$. Therefore, 
         \begin{align*}
             \varphi^{\prime}(t)&<0, \,\, \mbox{for}\,\, t\in(0,t_0)\\
             &\,\,\,\,\,\,\,\,\,\,\,\,\,\,\mbox{and}\\
             \varphi^{\prime}(t)&>0, \,\, \mbox{for}\,\, t\in(t_0,1).
         \end{align*}
         Thus, $\varphi(t)$ attains its global minimum at $t_0$. Solving the equation $\varphi^{\prime}(t)=0$, we obtain the roots $\frac{1\pm \sqrt{q}}{\sqrt{1-q}}$. Among these, only $t_0=\frac{1- \sqrt{q}}{\sqrt{1-q}} $ belongs to the interval $(0,1)$. Hence the global minimum value of $\varphi$ is $\frac{\sqrt{q}-1}{\sqrt{q}+1}$.\\

      Now define another function $\psi:(0,1)\to \mathbb{R}$  by    \begin{align*}
         \psi(t)=\gamma^{+}_{\lambda}t^2=\frac{(1-a^2)t^2+at(1-t^2)}{1-a^2t^2}.
     \end{align*} We have \begin{align*}
         \psi^{\prime}(t)=\frac{(1-at)^2(at^2+2t+a)}{(1-a^2t^2)^2}>0,\,\,\,\mbox{for all } \,\, t\in(0,1).
     \end{align*} Thus, $\psi$  is strictly increasing on $(0,1)$, and therefore  $\sup_{0<t<1}\psi(t)=1$.
\end{proof}

A bounded linear operator $T:A^2(\mathbb{D})\to A^2(\mathbb{D})$ is called a finite-rank operator if dim$(T(A^2(\mathbb{D}))<\infty$. For any $f\in  A^2(\mathbb{D})$, a finite-rank operator $T$ admits the representation
\[T(f)=\sum_{i=0}^n\langle f,g_i\rangle h_i,\,\,\, \mbox{for some } g_i,h_i\in A^2(\mathbb{D}).\] It was established in \cite{atul} that the Berezin range associated with such a finite-rank operator on $A^2(\mathbb{D})$ is symmetric about the real axis. Motivated by this result, we next investigate the following analogous geometric property of the $q$-Berezin range for finite-rank operators on   $A^2(\mathbb{D})$.

 \begin{theorem}\label{symmetric}
     Let $T(f)=\sum_{i=0}^n\langle f,g_i\rangle h_i$, where $f,g_i,h_i\in A^2(\mathbb{D})$ for all $i\in \{0,1,\cdots, n\}$ be  a finite rank operator on $A^2(\mathbb{D})$ and $q\in (0,1)$ such that $g_i(z)=\sum_{m=0}^{\infty}a_{i,m}z^m$ and $h_i(z)=\sum_{m=0}^{\infty}b_{i,m}z^m$. If $a_{i,m},b_{i,m}\in \mathbb{R}$, then the $q$-Berezin range of $T$, $\textbf{Ber}_q(T)$ is symmetric about real axis.
 \end{theorem}
 \begin{proof}
     Let $\lambda=r_1e^{i\theta_1}\in \mathbb{D}$, and $\mu\in S_{\lambda}$, i.e., $\langle \hat{k}_{\lambda},\hat{k}_{\mu}\rangle=q$. Then 
     \begin{align*}
         \langle T\hat{k}_{\lambda},\hat{k}_{\mu}\rangle&=(1-|\lambda|^2)(1-|\mu|^2)\langle Tk_{\lambda}, k_{\mu}\rangle\\
         &=(1-|\lambda|^2)(1-|\mu|^2)\Big\langle \sum_{i=0}^n\langle k_{\lambda},g_i\rangle h_i, k_{\mu}\Big\rangle\\
           &=(1-|\lambda|^2)(1-|\mu|^2)\left( \sum_{i=0}^n \overline{g_i(\lambda)} h_i(\mu)\right)\\
             &=q(1-\overline{\lambda}\mu)^2 \sum_{i=0}^n\left( \overline{\sum_{m=0}^{\infty}a_{i,m}\lambda^m} \sum_{m=0}^{\infty}b_{i,m}\mu^m\right).
     \end{align*}
     We have to show that $  \langle T\hat{k}_{\lambda},\hat{k}_{\mu}\rangle=\overline{  \langle T\hat{k}_{\overline{\lambda}},\hat{k}_{\overline{\mu}}\rangle}$. We consider the following cases.\\
     Case 1: For $\lambda=0$, then 
     \begin{align*}
          \langle T\hat{k}_{\lambda},\hat{k}_{\mu}\rangle&=q\sum_{i=0}^n\left(a_{i,0}\sum_{m=0}^{\infty}b_{i,m}\mu^m\right)\\
&=\overline{q\sum_{i=0}^n\left(a_{i,0}\sum_{m=0}^{\infty}b_{i,m}\overline{\mu}^m\right)}=\overline{  \langle T\hat{k}_{\overline{\lambda}},\hat{k}_{\overline{\mu}}\rangle}.\end{align*}
     Case 2: For $\lambda\neq 0$,  $\mu=p_{\lambda}\lambda$, where $p_{\lambda}\in \{\gamma_{\lambda}^+,\gamma_{\lambda}^-\}$. Then
     \begin{align*}
         \overline{  \langle T\hat{k}_{\overline{\lambda}},\hat{k}_{\overline{\mu}}\rangle}&=\overline{q\big(1-p_{\lambda} |\lambda|^2\big)^2\sum_{i=0}^{n}\left(\overline{\sum_{m=0}^{\infty}a_{i,m}\big(r_1e^{-i\theta_1}\big)^m}\sum_{m=0}^{\infty}b_{i,m}\big(p_{\lambda}r_1e^{-i\theta_1}\big)^m\right)}\\
         &=q\big(1-p_{\lambda} |\lambda|^2\big)^2\sum_{i=0}^{n}\left(\sum_{m=0}^{\infty}a_{i,m}\big(r_1e^{-i\theta_1}\big)^m\overline{\sum_{m=0}^{\infty}b_{i,m}\big(p_{\lambda} r_1e^{-i\theta_1}\big)^m}\right)\\
         &=q\big(1-p_{\lambda} |\lambda|^2\big)^2\sum_{i=0}^{n}\left(\overline{\sum_{m=0}^{\infty}a_{i,m}\big(r_1e^{i\theta_1}\big)^m}\sum_{m=0}^{\infty}b_{i,m}\big(p_{\lambda} r_1e^{i\theta_1}\big)^m\right)\\
         &=\langle T\hat{k}_{\lambda},\hat{k}_{\mu}\rangle.
     \end{align*} Hence, the proof is complete.
 \end{proof}
 \begin{cor}
     Let $T(f)=\sum_{i=0}^n\langle f,g_i\rangle h_i$, where $f,g_i,h_i\in A^2(\mathbb{D})$ for all $i\in\{0,1,\cdots, n\}$ be a finite rank operator on $A^2(\mathbb{D})$ and $q\in (0,1)$. If $g_i(z)=\sum_{m=0}^{\infty}a_{i,m}z^m$ and $h_i(z)=\sum_{m=0}^{\infty}b_{i,m}z^m$ such that $a_{i,m},b_{i,m}\in \mathbb{R}$ and also the $q$-Berezin range of $T$ on $ A^2(\mathbb{D})$ is convex, then $\Re\{\langle T\hat{k}_{\lambda},\hat{k}_{\mu}\rangle\}\in \textbf{Ber}_q(T)$ for every $\lambda\in \mathbb{D}$ and $\mu\in S_{\lambda}$. 
 \end{cor}
 \begin{proof}
     Suppose that the $q$-Berezin range $\textbf{Ber}_q(T)$ is convex. By Theorem \ref{symmetric}, $\textbf{Ber}_q(T)$ is symmetric with respect to the real axis. Thus, we have 
     \begin{align*}
         \frac{1}{2}\langle T\hat{k}_{\lambda},\hat{k}_{\mu}\rangle+\frac{1}{2}  \langle T\hat{k}_{\overline{\lambda}},\hat{k}_{\overline{\mu}}\rangle= \Re\{\langle T\hat{k}_{\lambda},\hat{k}_{\mu}\rangle\}\in \textbf{Ber}_q(T).
     \end{align*}\end{proof}

    A rank one operator $T$ on $A^2(\mathbb{D})$ is of the form $T(f)=\langle f ,g \rangle h$ for some $g,h \in A^2(\mathbb{D})$. In this case 
    \begin{align*}
        \langle T\hat{k}_{\lambda},\hat{k}_{\mu}\rangle=
        \langle \hat{k}_{\lambda},g\rangle\langle h,\hat{k}_{\mu}\rangle =\overline{g(\lambda)}h(\mu).
    \end{align*} Therefore, the $q$-Berezin range of $T$ is given by\[\textbf{Ber}_q(T)=\big\{\overline{g(\lambda)}h(\mu):\lambda\in \mathbb{D}, \mu \in S_{\lambda}\big\}.\]
In \cite{atul}, the authors proved that the Berezin range of the rank-one operator of the form $T(f)=\langle f,z^n\rangle z^m$, where $f\in A^2(\mathbb{D})$ and $m,n\in \mathbb{N}$, is convex.
In the following result, we establish that, for $q\in(0,1)$, the $q$-Berezin range of the same rank-one operator is a circular disc centered at the origin.

    \begin{theorem}
        Let $T(f)=\langle f,z^n\rangle z^m$, where $f\in A^2(\mathbb{D})$, $m,n\in \mathbb{N}$  with $m>n$  and $q\in (0,1)$. Then $\textbf{Ber}_q(T)$ is a disc centered at the origin. 
    \end{theorem}
    \begin{proof} Let $\lambda=|\lambda|e^{i\theta}\in \mathbb{D}$ and $\mu \in S_{\lambda}$, i.e.,  $\langle \hat{k}_{\lambda},\hat{k}_{\mu}\rangle=q$. Then   
          \begin{align*}
         \langle T\hat{k}_{\lambda},\hat{k}_{\mu}\rangle&=(1-|\lambda|^2)(1-|\mu|^2)\langle Tk_{\lambda}, k_{\mu}\rangle\\
         &=q(1-\overline{\lambda}\mu)^2\left\langle \langle k_{\lambda},z^n\rangle z^m, k_{\mu}\right\rangle\\
          &=q(1-\overline{\lambda}\mu)^2\langle \overline{\lambda}^n z^m, k_{\mu}\rangle\\
            &=q(1-\overline{\lambda}\mu)^2\overline{\lambda}^n\mu^m.
         \end{align*}
         If $\lambda=0$, then clearly $\langle T\hat{k}_{\lambda},\hat{k}_{\mu}\rangle=0$.\\
          For  $\lambda\neq 0$, we have $\mu=p_{\lambda} \lambda$, where $p_{\lambda}\in \{\gamma_{\lambda}^+,\gamma_{\lambda}^-\}$. 
We first consider the case where $p_{\lambda}=\gamma_{\lambda}^+$.  Then
         \begin{align*} \langle T\hat{k}_{\lambda},\hat{k}_{\mu}\rangle&=q(1-\gamma_{\lambda}^+|\lambda|^2)^2(\gamma_{\lambda}^+)^m|\lambda|^{m+n}e^{i(m-n)\theta}\\
                 &=\nabla_+(|\lambda|)e^{i(m-n)\theta},
         \end{align*}  where the function $\nabla_+: (0,1)\to \mathbb{R}$ is defined by \[\nabla_+(|\lambda|)=q(1-\gamma_{\lambda}^+|\lambda|^2)^2(\gamma_{\lambda}^+)^m|\lambda|^{m+n}.\]Similarly, when $p_{\lambda}= \gamma_\lambda^-$, we obtain  $\langle T\hat{k}_{\lambda},\hat{k}_{\mu}\rangle=\nabla_-(|\lambda|)e^{i(m-n)\theta}$, where \[\nabla_-(|\lambda|)=q(1-\gamma_{\lambda}^-|\lambda|^2)^2(\gamma_{\lambda}^-)^m|\lambda|^{m+n}.\]
         Therefore, 
         \begin{align*}
             \textbf{Ber}_q(T)=&\Big\{\nabla_+(|\lambda|)e^{i(m-n)\theta}, 0< |\lambda|<1, 0\le \theta<2\pi\Big\}\nonumber\\&\cup \Big\{\nabla_-(|\lambda|)e^{i(m-n)\theta}, 0< |\lambda|<1, 0\le \theta<2\pi\Big\}.
         \end{align*} 
           Let $\eta(\neq 0)\in  \textbf{Ber}_q(T)$. Then $\eta$ be equal to either \begin{align*}
                \nabla_+(|\lambda_1|)e^{i(m-n)\theta_1} \,\,\,\mbox{or}\,\,\, \nabla_-(|\lambda_2|)e^{i(m-n)\theta_2}
           \end{align*}
           for some $0<|\lambda_1|,|\lambda_2|<1$ and $0\le \theta_1,\theta_2<2\pi$. Clearly, for any $\xi\in [0,2\pi)$, $\eta e^{i\xi}$  also belongs to $\textbf{Ber}_q(T)$. Moreover, $\lim_{|\lambda|\to 0}\big|\nabla_+(|\lambda|)\big|=0 $ and $\lim_{|\lambda|\to 0}\big|\nabla_-(|\lambda|)\big|=0 $. Since $0\in \textbf{Ber}_q(T) $, it follows that $  \textbf{Ber}_q(T)$  is a disc with center at origin. 
     \end{proof}

Let $T$ be a bounded linear operator on $A^2(\mathbb{D})$. Suppose that $T$ is diagonal with respect to the orthonormal basis $\{\sqrt{n+1}z^n\}_{n=0}^{\infty}$. Then $T$ admits the representation $$T(f)=\sum_{n=0}^{\infty}\alpha_n(n+1)\langle f, z^n\rangle z^n,$$ where $\{\alpha_n\}_{n=0}^{\infty}$ is a bounded sequence. The following theorem establishes that the $q$-Berezin range of such diagonal operators is convex.

     \begin{theorem}
          Let $T(f)=\sum_{n=0}^{\infty}\alpha_n (n+1)\langle f, z^n\rangle z^n$, where $f\in A^2(\mathbb{D})$, $\{\alpha_n\}_{n=0}^{\infty}$ is a bounded sequence of real numbers, be an operator on $A^2(\mathbb{D})$ and $q\in (0,1]$. Then the $q$-Berezin range of $T$ is convex.
     \end{theorem}
     \begin{proof}
        Let $\lambda\in \mathbb{D}$ and $\mu \in S_{\lambda}$, i.e.,  $\langle \hat{k}_{\lambda},\hat{k}_{\mu}\rangle=q$. Then  
        \begin{align*}
         \langle T\hat{k}_{\lambda},\hat{k}_{\mu}\rangle&=(1-|\lambda|^2)(1-|\mu|^2)\langle Tk_{\lambda}, k_{\mu}\rangle\\
         &=q(1-\overline{\lambda}\mu)^2\left\langle \sum_{n=0}^{\infty}\alpha_n(n+1)\langle  k_{\lambda}, z^n\rangle z^n,k_{\mu}\right\rangle\\
          &=q(1-\overline{\lambda}\mu)^2\sum_{n=0}^{\infty}\alpha_n(n+1) (\overline{\lambda}\mu)^n
         \end{align*} In the case $\lambda=0$,  we get $ \langle T\hat{k}_{\lambda},\hat{k}_{\mu}\rangle=q\alpha_0$. By Lemma \ref{oopp}, it follows that \begin{align*}
             \textbf{Ber}_q(T)&=\Big\{q(1-t)^2\sum_{n=0}^{\infty}\alpha_n(n+1) t^n:t\in \left[\frac{\sqrt{q}-1}{\sqrt{q}+1},1\right)\Big\},\,\,\mbox{for}\,\, q\in (0,1)\\
             \mbox{and}\,\,\, \textbf{Ber}(T)&=\Big\{(1-t^2)^2\sum_{n=0}^{\infty}\alpha_n(n+1) t^{2n}:0\le t<1\Big\},\,\,\mbox{for} \,\,q=1.
         \end{align*} Thus, in both cases, $\textbf{Ber}_q(T)$ is the image of an interval under a continuous function. Since the image of a connected set under a continuous map remains connected, and every connected subset of $\mathbb{R}$ is convex, we conclude that $\textbf{Ber}_q(T)$ is convex.
         \end{proof}
        The Berezin range of composition operators acting on the Hardy-Hilbert space and Bergman space has been studied in \cite{Composition,Bulletin des}.  Next, we investigate the convexity of the $q$-Berezin range, for $q\in(0,1)$,  associated with
composition operators on the Bergman space $A^2(\mathbb{D})$. Let $\phi:\mathbb{D}\to \mathbb{D}$ be a complex-valued function.  The composition operator $C_{\phi}$ acting on $A^2(\mathbb{D})$ is defined by $C_{\phi}f=f\circ\phi$.

First, we focus on the case $\phi(z)=\zeta z$ with $\zeta\in \overline{\mathbb{D}}$. Let $\lambda\in \mathbb{D}$ and $\mu\in S_{\lambda}$, i.e., $\langle k_{\lambda}, k_{\mu}\rangle=q$. Then \begin{align*}
    \langle C_{\phi} \hat{k}_{\lambda},\hat{k}_{\mu}\rangle=(1-|\lambda|^2)(1-|\mu|^2) C_{\phi}k_{\lambda}(\mu)=\frac{q(1-\overline{\lambda}\mu)^2}{(1-\zeta\overline{\lambda}\mu)^2}. 
\end{align*} In the particular case, if $\lambda=0$, then $ \langle C_{\phi} \hat{k}_{\lambda},\hat{k}_{\mu}\rangle=q$. Therefore,  for $q\in (0,1)$, it follows  from Lemma  \ref{oopp}
 that
 \begin{align}
    \textbf{ Ber}_q(C_{\phi})=\left\{\frac{q(1-t)^2}{(1-\zeta t)^2}: t\in \left[ \frac{\sqrt{q}-1}{\sqrt{q}+1},1\right)\right\}.\label{maa5}
 \end{align} 
 \begin{lemma}\label{qqpp}
     Let  $q\in(0,1)$ and  $\phi(z)=\zeta z$ where $\zeta\in \overline{\mathbb{D}}$. Then 
\begin{align*}
    &(i)~~  \textbf{ Ber}_q(C_{\phi})\,\, \mbox{  is a singleton if and only if } \,\, \zeta=1.\\
    &(ii)~~ \Im( \textbf{ Ber}_q(C_{\phi}))=\{0\}\,\, \mbox{if and only if } \,\, \Im(\zeta)=0.
\end{align*}
 \end{lemma}
 \begin{proof} 
     $(i)$~~ Since $\zeta=1$ trivially implies $\textbf{Ber}_q(C_\phi)$ is  singleton. Thus, it remains to prove that   $\textbf{Ber}_q(C_\phi)$ is singleton when $\zeta=1$.
    Let $ \textbf{ Ber}_q(C_{\phi})$ is a
singleton. Since $  \frac{q(1-t)^2}{(1-\zeta t)^2}=q$ at $t=0$ then we have, $ \textbf{ Ber}_q(C_{\phi})=\{q\}$, and consequently
     \begin{align*}
         \frac{q(1-t)^2}{(1-\zeta t)^2}=q\,\,\, \mbox{ holds  for all}\,\,  t\in \left[ \frac{\sqrt{q}-1}{\sqrt{q}+1},1\right)
     \end{align*} which holds only when $\zeta=1$.\\
   $(ii)$~~ Let $\zeta=a+ib$. Then 
   \begin{align*}
       \Im( \textbf{ Ber}_q(C_{\phi}))=\Im\left(\frac{q(1-t)^2}{(1-\zeta t)^2}\right)  =\frac{2q(1-t^2)bt(1-at)}{((1-at)^2+(bt)^2)^2},
   \end{align*}for all $t\in \left[ \frac{\sqrt{q}-1}{\sqrt{q}+1},1\right)$. From this, it is clear that $\Im( \textbf{ Ber}_q(C_{\phi}))=\{0\}$ if and only if $\Im(\zeta)=0$.
 \end{proof}

 \begin{theorem}
     Let $\zeta\in \overline{\mathbb{D}}$, $\phi(z)=\zeta z$ and $q\in (0,1)$. Then the $q$-Berezin range of $C_{\phi}$ acting on $A^2(\mathbb{D})$ is convex if and only if $\zeta\in [-1,1]$.
 \end{theorem}
 \begin{proof} 
      Let $\lambda\in \mathbb{D}$ and $\mu\in S_{\lambda}$, i.e.,  $\langle \hat{k}_{\lambda},\hat{k}_{\mu}\rangle=q$. Then  from equation \eqref{maa5}  we get, 
      \begin{align*}
             \textbf{ Ber}_q(C_{\phi})=\left\{\frac{q(1-t)^2}{(1-\zeta t)^2}: t\in \left[ \frac{\sqrt{q}-1}{\sqrt{q}+1},1\right)\right\}.
      \end{align*} Suppose $\zeta=1$ then $\textbf{Ber}_q(C_{\phi})=\{q\}$, which is convex. Similarly for $-1\le \zeta<1$, $\frac{q(1-t)^2}{(1-\zeta  t)^2}$ is decreasing function in $t$. Now $\lim_{t\to 1^-}\frac{q(1-t)^2}{(1-\zeta t)^2}=0$ and \begin{align*}
          \max\left\{\frac{q(1-t)^2}{(1-\zeta t)^2}: t\in \left[\frac{\sqrt{q}-1}{\sqrt{q}+1},1\right)\right\}=\frac{4q}{(\sqrt{q}+1-\zeta\sqrt{q}+\zeta)^2}
      \end{align*}  attained at  $t=\frac{\sqrt{q}-1}{\sqrt{q}+1}$. Hence $\textbf{Ber}_q(C_{\phi})=\Big(0,\frac{4q}{(\sqrt{q}+1-\zeta\sqrt{q}+\zeta)^2}\Big]$, which is also convex.

      Conversely,  let \begin{align*}
    \textbf{ Ber}_q(C_{\phi})=\left\{\frac{q(1-t)^2}{(1-\zeta t)^2}: t\in \left[ \frac{\sqrt{q}-1}{\sqrt{q}+1},1\right)\right\}
 \end{align*} is convex. Clearly, $\textbf{ Ber}_q(C_{\phi})$ is just a path in $\mathbb{C}$. Also, by convexity, it is a line segment or a singleton set. By Lemma \ref{qqpp} $(i)$,  $\textbf{ Ber}_q(C_{\phi})=\{q\}$ if and only if $\zeta=1$. Also, for $\zeta\neq 1$ , $\frac{q(1-t)^2}{(1-\zeta t)^2}=q$ at $t=0$ and $\lim_{t\to 1^-}\frac{q(1-t)^2}{(1-\zeta t)^2}=0$. Thus $\textbf{ Ber}_q(C_{\phi})$ is a line segment passing through the point $q$ and approaching the origin. Also Lemma \ref{qqpp} $(ii)$, we have $\Im(\textbf{ Ber}_q(C_{\phi}))=\{0\}$ if and only if $\Im(\zeta)=0$. As $\zeta\in\overline{ \mathbb{D}}$, we have $\zeta\in [-1,1]$.
 \end{proof}
For $\alpha\in \mathbb{D}$, consider the automorphism of the unit disc $\phi_{\alpha}(z)=\frac{z-\alpha}{1-\overline{\alpha}z}$ and the composition operator $C_{\phi_{\alpha}}$ acting on $A^2(\mathbb{D})$ is defined by $C_{\phi_{\alpha}}f=f\circ \phi_{\alpha}$. The symmetry of the Berezin range of $C_{\phi_{\alpha}}$ on $A^2(\mathbb{D})$ with respect to the real axis was investigated in \cite{Composition}. Here, we study the corresponding symmetry properties of the q-Berezin range for $q\in (0,1)$.
\begin{theorem}\label{rrcc}
    Let $q\in (0,1)$. Then the $q$-Berezin range of $C_{\phi_{\alpha
    }}$ on $A^2(\mathbb{D})$ is symmetric with respect to the real axis.
\end{theorem}
\begin{proof}
        Let $\lambda\in \mathbb{D}$ and $\mu \in S_{\lambda}$, i.e.,  $\langle \hat{k}_{\lambda},\hat{k}_{\mu}\rangle=q$ and $\alpha=\rho e^{i\psi}$.  Then  
        \begin{align*}
            \langle C_{\phi_{\alpha}}\hat{k}_{\lambda},\hat{k}_{\mu} \rangle=&(1-|\lambda|^2)(1-|\mu|^2) \langle C_{\phi_{\alpha}}{k}_{\lambda},{k}_{\mu} \rangle\\
            =& \frac{q(1-\overline{\lambda}\mu)^2}{\big(1-\overline{\lambda}\phi_{\alpha}(\mu)\big)^2}.
        \end{align*} 
        We have to prove 
        \begin{align*}
              \langle C_{\phi_{\alpha}}\hat{k}_{\lambda},\hat{k}_{\mu} \rangle=\overline{  \langle C_{\phi_{\alpha}}\hat{k}_{\overline{\lambda}e^{2i\psi}},\hat{k}_{\overline{\mu}e^{2i\psi}} \rangle}.
        \end{align*} Let $\lambda^{\prime}=\overline{\lambda}e^{2i\psi}$ and $\mu^{\prime}=\overline{\mu}e^{2i\psi}$. Then
\begin{align*}
    \overline{  \langle C_{\phi_{\alpha}}\hat{k}_{\overline{\lambda}e^{2i\psi}},\hat{k}_{\overline{\mu}e^{2i\psi}} \rangle}&=\overline{  \langle C_{\phi_{\alpha}}\hat{k}_{\lambda^{\prime}},\hat{k}_{\mu^{\prime}} \rangle}\\
      =& \frac{q\overline{(1-\overline{\lambda^{\prime}}\mu^{\prime})^2}}{\overline{\big(1-\overline{\lambda^{\prime}}\phi_{\alpha}(\mu^{\prime})\big)^2}}\\
        =& \frac{q(1-\overline{\lambda}\mu)^2}{\big(1-\overline{\lambda}\phi_{\alpha}(\mu)\big)^2}\\
        =&  \langle C_{\phi_{\alpha}}\hat{k}_{\lambda},\hat{k}_{\mu} \rangle.
\end{align*}
        So, our desired result follows.
\end{proof}
 \begin{cor}
     Let $q\in (0,1)$ and $\alpha\in\mathbb{D} $. If the $q$-Berezin range of $C_{\phi_{\alpha}}$ on $A^2(\mathbb{D})$ is convex, then $\Re\{C_{\phi_{\alpha}}\hat{k}_{\lambda},\hat{k}_{\mu}\}\in \textbf{Ber}_q(C_{\phi_{\alpha}})$ for each $\lambda\in \mathbb{D}$ and $\mu \in S_{\lambda}$.
 \end{cor}
 \begin{proof}
     Let $\textbf{Ber}_q(C_{\phi_{\alpha}})$ is convex. Now, from Theorem \ref{rrcc}, as $\textbf{Ber}_q(C_{\phi_{\alpha}})$ is symmetric about real line, then we have \begin{align*}
         \frac{1}{2}   \langle C_{\phi_{\alpha}}\hat{k}_{\lambda},\hat{k}_{\mu} \rangle+\frac{1}{2}  \langle C_{\phi_{\alpha}}\hat{k}_{\overline{\lambda}e^{2i\psi}},\hat{k}_{\overline{\mu}e^{2i\psi}} \rangle=\Re\{C_{\phi_{\alpha}}\hat{k}_{\lambda},\hat{k}_{\mu}\}\in \textbf{Ber}_q(C_{\phi_{\alpha}}).
     \end{align*}
 \end{proof}
 
In \cite{Bulletin des}, the authors have studied the Berezin range of certain weighted shift operators acting on the Hardy–Hilbert space and the Bergman space. Here, we investigate the $q$-Berezin range of weighted shift operators acting on the Bergman space.\\
Let $f(z)=\sum^{\infty}_{n=0}a_nz^n$ belong to the Bergman  $A^2(\mathbb{D})$. The weighted shift operator $T_{\beta}$ acting on $A^2(\mathbb{D})$ is defined as $$T_{\beta}\left(\sum_{n=0}^{\infty}a_nz^n\right)=\sum_{n=0}^{\infty}a_n\beta_{n}z^{n+1},$$
where $\beta=\{\beta_n\}$ is a  sequence in $\mathbb{C}.$ If the sequence $\{\beta_n\}$ is bounded, there exists a constant $K\in\mathbb{R}$ such that $|\beta_n| \leq K,~~ n\in \mathbb N.$ Consequently, we have
  $$\sum_{n=0}^{\infty}\frac{|a_n\beta_{n}|^2}{n+2}\leq K^2\sum_{n=0}^{\infty}\frac{|a_n|^2}{n+1}<\infty.$$ 
  Thus, $T_{\beta}$ defines a bounded linear operator on $A^2(\mathbb{D}).$ In the following result, we consider certain specific sequences and study the convexity and  the symmetry of the $q$-Berezin range of 
$T_{\beta}$ about the real and imaginary axes. 
\begin{theorem}
    Let $q\in(0,1]$ and $T_\beta$ be a weighted shift operator on $A^2(\mathbb{D})$, where $\beta=\{\beta_n\}_{n=0}^{\infty}$ be a bounded real sequence. Then $\textbf{Ber}_q(T_\beta)$ is a disc centered at the origin, hence convex.\label{kpp}
\end{theorem}

\begin{proof}
     Let $\lambda=|\lambda|e^{i\theta}\in \mathbb{D}$ and $\mu \in S_{\lambda}$, i.e.,  $\langle \hat{k}_{\lambda},\hat{k}_{\mu}\rangle=q$. Then  
        \begin{align}
         \langle T_\beta\hat{k}_{\lambda},\hat{k}_{\mu}\rangle&=(1-|\lambda|^2)(1-|\mu|^2)\langle T_{\beta}k_{\lambda}, k_{\mu}\rangle\nonumber\\
         &=q(1-\overline{\lambda}\mu)^2\mu\sum_{n=0}^{\infty}(n+1)\beta_n(\overline{\lambda}\mu)^n.\label{suva2}\end{align} 
         Let $q\in(0,1)$. If $\lambda=0$, then $\langle T_\beta\hat{k}_{\lambda},\hat{k}_{\mu}\rangle=q\beta_0\mu\in q\beta_0\mathbb{T}_{\sqrt{1-q}} $.

        For $\lambda\neq 0$,  we have $\mu=p_{\lambda} \lambda$, where $p_{\lambda}\in \{\gamma_{\lambda}^+,\gamma_{\lambda}^-\}$.  Assume first that $p_{\lambda}=\gamma_{\lambda}^+$.  Then \begin{align*}
             \langle T_\beta\hat{k}_{\lambda},\hat{k}_{\mu}\rangle&=q(1-\gamma_{\lambda}^+|\lambda|^2)^2\gamma_{\lambda}^+|\lambda|\sum_{n=0}^{\infty}(n+1)\beta_n(\gamma_{\lambda}^+|\lambda|^2)^n e^{i\theta}\\
             &=\nabla_{+}(|\lambda|)e^{i\theta},
         \end{align*}
 where the function $\nabla_+: (0,1)\to \mathbb{R}$ are defined by \begin{align*}
             \nabla_+(|\lambda|)=q(1-\gamma_{\lambda}^+|\lambda|^2)^2\gamma_{\lambda}^+|\lambda|\sum_{n=0}^{\infty}(n+1)\beta_n(\gamma_{\lambda}^+|\lambda|^2)^n.
         \end{align*} 
      
Similarly, when $p_{\lambda}= \gamma_\lambda^-$, we obtain $ \langle T_\beta\hat{k}_{\lambda},\hat{k}_{\mu}\rangle=\nabla_{-}(|\lambda|)e^{i\theta}$, where 
         \begin{align*}
             \nabla_-(|\lambda|)=q(1-\gamma_{\lambda}^-|\lambda|^2)^2\gamma_{\lambda}^-|\lambda|\sum_{n=0}^{\infty}(n+1)\beta_n(\gamma_{\lambda}^-|\lambda|^2)^n.
         \end{align*} 
          Now, we have 
         \begin{align*}
             \textbf{Ber}_q(T_\beta)=&\left\{\nabla_+(|\lambda|)e^{i\theta}:0<|\lambda|<1, 0\le \theta<2\pi\right\}\\
           & \cup \left\{\nabla_-(|\lambda|)e^{i\theta}:0<|\lambda|<1, 0\le \theta<2\pi\right\}\cup q\beta_0\mathbb{T}_{\sqrt{1-q}}. 
         \end{align*}
         
         Let $\eta(\neq 0)\in  \textbf{Ber}_q(T_\beta)$ then $\eta$ be equal to either $\nabla_+(|\lambda_1|)e^{i\theta_1}$ or $\nabla_-(|\lambda_2|)e^{i\theta_2}$ for some $0<|\lambda_1|,|\lambda_2|<1$ and $0\le \theta_1,\theta_2<2\pi$.  Now, for any $\xi\in [0,2\pi)$,
         $\eta e^{\xi}\in\textbf{Ber}_q(T_\beta)$. Moreover, $\lim_{|\lambda|\to 0^+}\nabla_+(|\lambda|)=q\beta_0\sqrt{1-q}$ and $\lim_{|\lambda|\to 0^+}\nabla_-(|\lambda|)=-q\beta_0\sqrt{1-q}$. It follows that $\textbf{Ber}_q(T_\beta)$ is a circular disc with center at origin. Therefore, $\textbf{Ber}_q(T_\beta)$ is convex.\\
      When $q=1$, we have $\lambda=\mu$. Thus from $\eqref{suva2}$ we get 
      \begin{align*}
          \textbf{Ber}_q(T_\beta)=\Big\{(1-|\lambda|^2)^2|\lambda|\sum_{n=0}^{\infty}(n+1)\beta_n|\lambda|^{2n}e^{i\theta}:0\le|\lambda|<1, 0\le \theta<2\pi\Big\}.
      \end{align*} Now, $\lim_{|\lambda|\to 0^+}(1-|\lambda|^2)^2|\lambda|\sum_{n=0}^{\infty}(n+1)\beta_n|\lambda|^{2n}=0$. Hence, the desired conclusion follows.
\end{proof}

\begin{theorem}
    Let $T_{\beta}$ be the weighted shift operator on $A^2(\mathbb{D})$,  where $\beta=\{\beta_n\}_{n=0}^{\infty}$ be a purely imaginary  bounded  sequence and $q\in(0,1]$. 
 Then $\textbf{Ber}_q(T_{\beta})$ is symmetric with respect to the imaginary axis. 
\end{theorem}

\begin{proof}
  Let $\lambda\in \mathbb{D}$ and $\mu\in S_{\lambda}$, i.e.,  $\langle \hat{k}_{\lambda},\hat{k}_{\mu}\rangle=q$. Then we have  
        \begin{align}
         \langle T_{\beta}\hat{k}_{\lambda},\hat{k}_{\mu}\rangle=q(1-\overline{\lambda}\mu)^2 \mu\sum_{n=0}^{\infty}(n+1)\beta_{n}(\overline{\lambda}\mu)^n.\label{suva}
         \end{align}  
        Let $q\in(0,1)$. If $\lambda=0$, then 
         \begin{align*}
              \{\langle T_{\beta}\hat{k}_{0},\hat{k}_{\mu}\rangle: \mu \in S_0\}=q \beta_0\mathbb{T}_{\sqrt{1-q}},
         \end{align*}
        
        which is symmetric about the imaginary axis.\\ Now, suppose that $\lambda\neq 0$. Then  $\mu=p_{\lambda} \lambda$, where $p_{\lambda}\in \{\gamma_{\lambda}^+,\gamma_{\lambda}^-\}$. Consequently,
         \begin{align*}
               \langle T_{\beta}\hat{k}_{\lambda},\hat{k}_{\mu}\rangle=q(1-p_{\lambda}|\lambda|^2)^2 p_{\lambda}\lambda\sum_{n=0}^{\infty}(n+1)\beta_{n}p_{\lambda}^n|\lambda|^{2n}.
         \end{align*}

        Let  $ \langle T_{\beta}\hat{k}_{\lambda},\hat{k}_{\mu}\rangle= a+ib$. Since $\{\beta_n\}$  is a bounded purely imaginary sequence, it follows that  $-a+ib=\langle T_{\beta}\hat{k}_{\overline{\lambda}},\hat{k}_{\overline{\mu}}\rangle.$ \\Thus, whenever $a+ib\in \textbf{Ber}_q(T_{\beta})$, the point $-a+ib$  also belongs to $\textbf{Ber}_q(T_{\beta})$.
        \\
        For $q=1$, we have $\lambda=\mu$. In this case,  the result follows similarly from \eqref{suva}.
     
\end{proof}
\section{Conclusion} This paper introduces $q$-Berezin sectorial operators and highlights the effectiveness of the $q$-Berezin sectorial framework through various inequalities involving the $q$-Berezin number.
The notion of $q$-Berezin sectorial operators also opens up several  promising directions  for future research, particularly in developing sharper bounds of the $q$-Berezin number.
This paper also investigates the $q$-Berezin range of several classes of bounded linear operators on the
Bergmann space. Future investigations may focus on extending these findings for other reproducing kernel Hilbert spaces.

\section*{Declarations}	
\textit{Acknowledgements.}
 Mr. Saikat Mahapatra 
would like to thank the UGC, Govt. of India, for the financial support (NTA Ref. No. 211610170555) in the form of a fellowship. Mr. Subhadip Halder would like to thank the UGC, Govt. of India, for the financial support (NTA Ref. No. 211610189204) in the form of a fellowship.\\
\textit{Author Contributions:} All the authors contributed equally to this manuscript and reviewed it. \\
 \textit{Data Availability:} No datasets were generated or analysed during the current study. \\
\textit{Conflict of interest:} There is no competing interest.\\

 \bibliographystyle{amsplain}

\end{document}